\documentclass[10pt]{article}
\def\version{June 14, 2013}

\nonstopmode
\usepackage[T1]{fontenc}
\usepackage[latin1]{inputenc}
\usepackage[english]{babel}

\usepackage[usenames,dvipsnames,dvips]{xcolor}
\definecolor{MyDarkBlue}{rgb}{0,0.08,0.45}
\usepackage[backref=none]{hyperref}
\hypersetup{pdfborder={0 0 0},
  colorlinks,
  urlcolor={MyDarkBlue},
  linkcolor={MyDarkBlue},
  citecolor={MyDarkBlue},
  breaklinks=true}
\usepackage{doi}
\providecommand{\url}[1]{\small\textcolor{blue}{#1}}
\providecommand{\eprint}[1]{}
\renewcommand{\eprint}[1]{arXiv:\href{http://arxiv.org/abs/#1}{#1}}

\DeclareSymbolFont{extraup}{U}{zavm}{m}{n}
\DeclareMathSymbol{\vardiamond}{\mathalpha}{extraup}{87}

\usepackage{amsthm}
\usepackage{bm}
\usepackage{upgreek}
\usepackage{mathrsfs}
\usepackage{times}
\usepackage{amssymb}
\usepackage{amsmath}

\providecommand{\ead}[1]{}
\providecommand{\journal}[1]{}

\DeclareSymbolFont{EUR}{U}{eur}{m}{n}
\SetSymbolFont{EUR}{bold}{U}{eur}{b}{n}
\DeclareSymbolFontAlphabet{\eur}{EUR}

\DeclareSymbolFont{EUB}{U}{eur}{b}{n}
\SetSymbolFont{EUB}{bold}{U}{eur}{b}{n}
\DeclareSymbolFontAlphabet{\eub}{EUB}

\DeclareSymbolFont{AMSb}{U}{msb}{m}{n}
\DeclareSymbolFontAlphabet{\mathbb}{AMSb}

\newcommand{\notyet}[1]{{}}

\newcommand{\eubD}{\eub{D}}
\newcommand{\eubL}{\eub{L}}
\newcommand{\eurL}{\eur{L}}
\newcommand{\eubJ}{\eub{J}}

\newcommand{\range}{\mathop{\rm Range\,}}

\newcommand{\p}{\partial}
\newcommand{\at}[1]{\vert\sb{\sb{#1}}}

\def\R{\mathbb{R}}
\newcommand{\C}{\mathbb{C}}

\newcommand{\N}{\mathbb{N}}

\newcommand{\abs}[1]{\vert #1 \vert}

\newcommand{\norm}[1]{\Vert #1 \Vert}

\newcommand{\sothat}{{\rm ;}\ }

\DeclareMathSymbol{\varGamma}{\mathord}{letters}{"00}
\DeclareMathSymbol{\varDelta}{\mathord}{letters}{"01}
\DeclareMathSymbol{\varTheta}{\mathord}{letters}{"02}
\DeclareMathSymbol{\varLambda}{\mathord}{letters}{"03}
\DeclareMathSymbol{\varXi}{\mathord}{letters}{"04}
\DeclareMathSymbol{\varPi}{\mathord}{letters}{"05}
\DeclareMathSymbol{\varSigma}{\mathord}{letters}{"06}
\DeclareMathSymbol{\varUpsilon}{\mathord}{letters}{"07}
\DeclareMathSymbol{\varPhi}{\mathord}{letters}{"08}
\DeclareMathSymbol{\varPsi}{\mathord}{letters}{"09}
\DeclareMathSymbol{\varOmega}{\mathord}{letters}{"0A}

\theoremstyle{plain}
\newtheorem{lemma}{Lemma}[section]
\newtheorem{theorem}[lemma]{Theorem}
\newtheorem{proposition}[lemma]{Proposition}

\theoremstyle{definition}
\newtheorem{definition}[lemma]{Definition}

\theoremstyle{remark}
\newtheorem{remark}[lemma]{Remark}

\renewenvironment{abstract}
  {\list{}{
  \setlength{\leftmargin}{1cm}
  \setlength{\rightmargin}{1cm}
  }%
  \item\relax \hskip 0pt {\it Abstract.}\ \small}
  {\endlist}

\makeatletter\@addtoreset{equation}{section}
\makeatletter\@addtoreset{lemma}{section}
\makeatother

\renewcommand{\Re}{\mathop{\rm{R\hskip -1pt e}}\nolimits}
\renewcommand{\Im}{\mathop{\rm{I\hskip -1pt m}}\nolimits}

\journal{Annales de l'Institut Henri Poincar\'e}

\begin{document}


\title{On linear instability of solitary waves
for the nonlinear Dirac equation
}

\author{
{\sc Andrew Comech}
\\
{\it\small Texas A\&M University, College Station, TX 77843, U.S.A.}
\\
{\it\small Institute for Information Transmission Problems,
Moscow 101447, Russia}
\\ \\
{\sc Meijiao Guan, \  Stephen Gustafson}
\\
{\it\small University of British Columbia,
Vancouver V6T 1Z2, Canada}
}

\date{\version}
\maketitle

\begin{abstract}
We consider the nonlinear Dirac equation,
also known as the Soler model:
\[
i\p\sb t\psi=-i\bm\alpha\cdot\bm\nabla\psi
+m\beta\psi
-(\psi\sp\ast\beta\psi)^k\beta\psi,
\quad
m>0,
\quad
\psi(x,t)\in\C^{N},
\quad
x\in\R^n,
\quad
k\in\N.
\]
We study the point spectrum
of linearizations at solitary waves
that bifurcate from NLS solitary waves
in the limit $\omega\to m$,
proving that if $k>2/n$,
then one positive and one negative eigenvalue
are present in the spectrum of the linearizations
at these solitary waves with $\omega$ sufficiently close
to $m$,
so that these solitary waves
are linearly unstable.
The approach is based on applying the Rayleigh--Schr\"odinger perturbation
theory to the nonrelativistic limit of the equation.
The results are in formal agreement with the
Vakhitov--Kolokolov stability criterion.

\medskip

\noindent
{\it R\'esum\'e.}
Nous consid\'erons l'\'equation de Dirac non lin\'eaire,
aussi connu comme le mod\`ele de Soler.
Nous \'etudions
le spectre ponctuel
des lin\'earisations
aux ondes solitaires des petites amplitudes 
dans la limite $\omega\to m$,
et montrons que
si $k>2/n$,
ensuite une valeur propre positive et une n\'egative
sont pr\'esents dans le spectre des lin\'earisations
\`a ces ondes solitaires
lorsque $\omega$ est suffisamment proche
de $m$,
ensuite ces ondes solitaires
sont lin\'eairement instable.
L'approche est bas\'ee sur l'application
de th\'eorie de la perturbation de Rayleigh--Schr\"odinger
\`a la limite non relativiste de l'\'equation.
Les r\'esultats sont en accord formel avec le
crit\`ere de stabilit\'e de Vakhitov--Kolokolov.
\end{abstract}


\section{Introduction}

A natural simplification
of the Dirac--Maxwell system
\cite{MR0190520}
is the nonlinear Dirac equation,
such as
the massive Thirring model \cite{MR0091788}
with vector-vector self-interaction
and
the Soler model \cite{PhysRevD.1.2766}
with scalar-scalar self-interaction
(known in dimension $n=1$
as the massive Gross--Neveu model
\cite{PhysRevD.10.3235,PhysRevD.12.3880}).
These models with self-interaction of local type
have been receiving a lot of attention in
particle physics (see e.g. \cite{rebbi1984solitons}),
as well as in the theory of Bose--Einstein
condensates \cite{MR2542748,PhysRevLett.104.073603}.

There is an enormous body of research
devoted to the nonlinear Dirac equation,
which we can not cover comprehensively here.
The existence of standing waves in the nonlinear Dirac equation was studied in
\cite{PhysRevD.1.2766},
\cite{MR847126}, \cite{MR949625}, and \cite{MR1344729}.
The question of stability
of solitary waves
is of utmost importance:
perturbations ensure
that we only ever encounter stable configurations.
Recent attempts at asymptotic stability
of solitary waves in the nonlinear Dirac equation
\cite{MR2259214,MR2466169,MR2985264,MR2924465,2012arXiv1203.2120C}
rely on the fundamental question
of \emph{spectral stability}:

\medskip
\noindent

\begin{verse}
{\it
Consider the Ansatz
$\psi(x,t)=(\phi\sb\omega(x)+\rho(x,t))e^{-i\omega t}$,
with $\phi\sb\omega(x)e^{-i\omega t}$ a solitary wave solution.
Let $\p\sb t\rho=A\sb\omega\rho$ be the linearized equation for $\rho$.
Does
$A\sb\omega$ have
eigenvalues
in the right half-plane?
}
\end{verse}

\medskip

\begin{definition}
If $\sigma(A\sb\omega)\subset i\R$,
we say that the solitary wave
$\phi\sb\omega e^{-i\omega t}$
is spectrally stable.
Otherwise, we say that the solitary wave is
linearly unstable.
\end{definition}

\begin{remark}
Let us note that the ``spectral stability''
$\sigma(A\sb\omega)\subset i\R$ does not guarantee stability.
One of the possibilities which may still lead to instability
is the presence of eigenvalues of higher algebraic multiplicity.
This occurs e.g. at $\lambda=0$
in the case $dQ(\omega)/d\omega=0$ 
(where $Q(\omega)$ is the charge of 
$\phi_\omega$-- see below) at a particular value of $\omega$,
leading to instability of the corresponding solitary wave;
see \cite{MR1995870}.
In \cite{MR2571965},
to reflect this situation in the context
of the nonlinear Schr\"odinger equation,
linear instability is defined as
either the presence of eigenvalues with positive real part
or the presence of particular
eigenvalues with higher algebraic multiplicity.
\end{remark}

\medskip


There is a very clear picture
of the spectral stability for
nonlinear Schr\"odinger and Klein--Gordon equations
\cite{VaKo,MR723756,MR783974,MR804458}
and general results for abstract Hamiltonian systems
with $\mathbf{U}(1)$ symmetry \cite{MR901236},
which stimulated
many attempts at spectral stability
in the nonlinear Dirac context.
We mention the numerical simulations
\cite{PhysRevLett.50.1230}
and the analysis of the energy minimization
under charge-preserving dilations
and similar transformations
\cite{MR592382,PhysRevD.34.644,MR848095,PhysRevE.82.036604}.
In spite of this,
the question
of spectral stability
of solitary waves of nonlinear Dirac equation
is still completely open.
Numerical results
\cite{MR2892774}
show that in the 1D Soler model (cubic nonlinearity)
all solitary waves are spectrally stable.
We also mention the related numerical results in
\cite{MR2217129,chugunova-thesis}.

According to \cite{VaKo},
if $\phi\sb\omega(x)e^{-i\omega t}$ is a family
of solitary wave solutions
to the nonlinear Schr\"odinger equation
\[
i\dot u=-\frac{1}{2m}\Delta u-f(\abs{u}^2)u,
\qquad
u(x,t)\in\C,
\quad
x\in\R^n,
\quad
n\ge 1,
\]
where $f$ is smooth and real-valued,
and if $\phi\sb\omega$ have
no nodes (such solitary waves are called \emph{ground states}),
then the linearization at the solitary wave
corresponding to a particular value of $\omega$
has a positive eigenvalue
if and only if at this value of $\omega$
one has
$dQ(\omega)/d\omega>0$, where
$Q(\omega)=\norm{\phi\sb\omega}\sb{L^2}^2$ is the
charge (or \emph{mass}) of the solitary wave.
The opposite condition,
\begin{equation}\label{vk}
\frac{d}{d\omega}Q(\omega)<0,
\end{equation}
is called the Vakhitov--Kolokolov stability criterion;
it ensures the absence of eigenvalues
with positive real part.
In the case of the nonlinear Dirac equation,
the condition \eqref{vk}
gives a less definite answer
about the spectral stability.
All we know is that
at the value of $\omega$
where
$\p\sb\omega Q(\omega)$ vanishes,
with $Q(\omega)$ being the charge
of the solitary wave
$\phi\sb\omega e^{-i\omega t}$,
two eigenvalues
of the linearized equation
collide at $\lambda=0$,
but we do not know where these eigenvalues
are located when $\p\sb\omega Q(\omega)\ne 0$
(see e.g. \cite{dirac-vk}).



Yet, it is natural to expect
that the condition \eqref{vk}
remains meaningful in the nonrelativistic limit,
as it was suggested in \cite{PhysRevE.82.036604}.
While it is a common practice
to obtain solitary wave solutions
for relativistic equations
as bifurcations
from the solitary waves
to the equation corresponding to the nonrelativistic limit
(see e.g. the review \cite{MR2434346}),
we show that the spectrum
of the linearization at a solitary wave
could also be learned from the nonrelativistic limit.
More precisely,
we develop the idea that the family of real eigenvalues
of the linearization
at a solitary wave
of the nonlinear Dirac equation
bifurcating from $\lambda=0$
is a deformed family of eigenvalues
of the linearization of the corresponding nonlinear Schr\"odinger equation.
As a result,
we prove that
if the Vakhitov--Kolokolov stability criterion \cite{VaKo}
guarantees linear instability
for the NLS,
then the same conclusion also holds for
solitary waves with $\omega\lesssim m$ in the nonlinear Dirac equation.
Let us mention that our results only apply
to the solitary wave solutions
which we obtain
from the solitary waves of the nonlinear Schr\"odinger
equation in the nonrelativistic limit of the nonlinear Dirac.


The model and the main results
are described in Section~\ref{sect-results}.
The necessary constructions
in the context of the nonlinear Schr\"odinger equation
are presented in Section~\ref{sect-nls}.
The existence and asymptotics of solitary waves
of the nonlinear Dirac equation
is covered in Section~\ref{sect-linearization}.
The main result (Theorem~\ref{main-theorem-dirac-vbk})
follows from
Lemma~\ref{lemma-xx-ff-3}
(existence of solitary wave
solutions and their asymptotics)
and 
Proposition~\ref{prop-k3} (presence of a positive eigenvalue
in the spectrum of the linearized operator),
which we prove using the Rayleigh--Schr\"odinger perturbation theory.

\section{Main result}
\label{sect-results}

We consider the nonlinear Dirac equation
\begin{equation}\label{nld-nd}
i\p\sb t\psi=-i\bm\alpha\cdot\bm\nabla\psi
+m\beta\psi
-f(\psi\sp\ast\beta\psi)\beta\psi,
\qquad
\psi(x,t)\in\mathbb{C}^N,
\qquad
x\in\R^n,
\end{equation}
with
$m>0$
and
$\psi\sp\ast$
being the Hermitian conjugate of $\psi$.
We assume that the nonlinearity $f(s)$
is smooth and real-valued,
and that
\begin{equation}\label{f-zero}
f(0)=0.
\end{equation}
Above,
$
\bm\alpha\cdot\bm\nabla
=\sum\limits\sb{j=1}\sp{n}
\alpha\sb j\frac{\p}{\p x\sb j},
$
and the Hermitian matrices $\alpha\sb j$ and $\beta$
are chosen so that
\[
(-i\bm\alpha\cdot\bm\nabla+\beta m)^2=(-\Delta +m^2)I\sb{N},
\]
where $I\sb{N}$ is the $N\times N$ unit matrix.
That is, $\alpha_j$ and $\beta$ are to satisfy
\begin{equation}
\label{eq:def-Dirac-matrices}
\alpha_j \alpha_k + \alpha_k \alpha_j
= 2 \delta_{jk} I\sb{N},
\qquad
\beta^2=I\sb{N};
\qquad
\alpha_j \beta+\beta\alpha_j=0.
\end{equation}
The generalized massive Gross--Neveu model
(the scalar-scalar case with $k>0$
in the terminology of \cite{PhysRevE.82.036604})
corresponds to the nonlinearity $f(s)=\abs{s}^k$.

According to the Dirac--Pauli theorem
(cf. \cite{1928RSPSA.117..610D,vanderwaerden-1932,MR1508031}
and \cite[Lemma 2.25]{thaller}),
the particular choice of the matrices $\alpha_j$ and $\beta$ does not matter:

\begin{lemma}[Dirac--Pauli theorem]
\label{lemma-pauli}
Let $n\in\N$.
For any sets of Dirac matrices
$\alpha_j$ $\beta$ and  $\tilde\alpha_j$, $\tilde\beta$
of the same dimension $N$,
with $1\le j\le n$,
there is a unitary matrix $S$
such that
\begin{equation}\label{pauli-odd}
\tilde\alpha_j=S^{-1} \alpha_j S,
\quad
1\le j\le n,
\qquad
\tilde\beta=S^{-1}\beta S
\end{equation}
if $n$ is odd,
and such that
\begin{equation}\label{pauli-even}
\tilde\alpha_j=\sigma S^{-1}\tilde\alpha_j S,
\quad
1\le j\le n,
\qquad
\tilde\beta=\sigma S^{-1}\beta S,
\qquad
\sigma=\pm 1,
\end{equation}
if $n$ is even.
\end{lemma}

For more details, see \cite{dirac-spectrum}.

Lemma~\ref{lemma-pauli}
allows one, by a simple change of variable, to transform the nonlinear Dirac equation,
changing the set of Dirac matrices.
Thus, when studying the spectral stability,
we can choose the Dirac matrices at our convenience.
We use the standard Pauli matrices,
\begin{equation}\label{def-pauli-matrices}
\sigma\sb 1=\left(\begin{matrix}0&\,\,1\\1&\,\,0\end{matrix}\right),
\qquad
\sigma\sb 2=\left(\begin{matrix}0&-i\\i&0\end{matrix}\right),
\qquad
\sigma\sb 3
=\left(\begin{matrix}1&0\\0&-1\end{matrix}\right),
\end{equation}
to make the following choice:
\begin{equation}\label{dm-1}
n=1: \quad N=2, \quad \alpha = -\sigma_2;
\end{equation}
\begin{equation}\label{dm-2}
n=2: \quad N=2, \quad \alpha\sb j = \sigma_j,
\quad
1\le j\le 2;
\end{equation}
\begin{equation}\label{dm-3}
n=3: \quad N=4, \quad \alpha\sb j
=\left(\begin{matrix}0&\sigma\sb j\\ \sigma\sb j&0\end{matrix}\right),
\quad
1\le j\le 3;
\end{equation}
and in all these cases we choose
\begin{equation}\label{dm-beta}
\beta
=\left(\begin{matrix}I\sb{N/2}&0\\0&-I\sb{N/2}\end{matrix}\right).
\end{equation}

\begin{remark}
If $n$ is even,
Lemma~\ref{lemma-pauli}
may only allow us to transform the equation \eqref{nld-nd}
with a particular set of the Dirac matrices
to the set of the Dirac matrices
as in \eqref{dm-1}, \eqref{dm-2}, \eqref{dm-3},
and \eqref{dm-beta},
but with the opposite signs
(this corresponds to $\sigma=-1$ in \eqref{pauli-even}).
In this case,
the signs of $\alpha_j$ are flipped by
taking the spatial reflections,
while $\beta$ being opposite to \eqref{dm-beta}
corresponds to considering the nonrelativistic limit
$\omega\to -m$,
with appropriate changes to
Theorem~\ref{main-theorem-dirac-vbk}.
\end{remark}

For a large class of nonlinearities $f(s)$,
there are solitary wave solutions
of the form
\begin{equation}\label{sw}
\psi(x,t)=\phi\sb\omega(x)e^{-i\omega t},
\qquad
\phi\sb\omega
\in H^1(\R,\C^{N}),
\qquad
\abs{\omega}<m.
\end{equation}
In dimension $n=1$,
one can take
\begin{equation}
\label{sol-forms-1}
\phi\sb\omega(x) =\begin{bmatrix}v(x,\omega)\\u(x,\omega)\end{bmatrix},
\end{equation}
with
$v(x,\omega)$ positive and even and $u(x,\omega)$ real-valued and odd;
under these conditions, the solitary wave
$\phi_\omega(x)$ is unique (see
Section~\ref{sect-solitary-waves-1d} for details).

For $n = 2$ and $n=3$, respectively,
\begin{equation}\label{sol-forms-23}
\phi\sb\omega(x) =\begin{bmatrix}v(r,\omega)\\
i e^{i\phi} u(r,\omega)\end{bmatrix},
\qquad
\phi\sb\omega(x) =
\begin{bmatrix}
v(r,\omega)
\begin{pmatrix}
1\\0
\end{pmatrix}
\\
i u(r,\omega)
\begin{pmatrix}
\cos\theta
\\
e^{i\phi}\sin\theta
\end{pmatrix}
\end{bmatrix},
\end{equation}
where $v(r,\omega)$ and
$u(r, \omega)$ are real-valued, radially-symmetric functions;
$(r,\phi)$ are standard polar coordinates in $\R^2$, and $(r,\theta,\phi)$ are
standard spherical coordinates in $\R^3$.
The existence of solitary waves of this form
is proved, for example, in~\cite{MR847126}.
The particular solutions we consider here,
however, are those constructed for $\omega$ close to $m$, as outlined in
Section~\ref{sect-sw-hd},
from the nonrelativistic limit $\omega \to m$.

Due to the $\mathbf{U}(1)$-invariance,
for solutions to \eqref{nld-nd}
the value of the charge functional
\[
Q(\psi)=\int\sb{\R^n}\abs{\psi(x,t)}^2\,dx
\]
is formally conserved.
For brevity,
we also denote by $Q(\omega)$
the charge of the solitary wave
$\phi\sb\omega(x)e^{-i\omega t}$:
\begin{equation}\label{def-q-omega}
Q(\omega)=\int\sb{\R^n}\abs{\phi\sb\omega(x)}^2\,dx.
\end{equation}

We are interested in the spectrum of linearization
of the nonlinear Dirac equation \eqref{nld-nd}
at a solitary wave solution \eqref{sw}.

\begin{theorem}\label{main-theorem-dirac-vbk}
Let $n\le 3$.
Assume that $f(s)=s^k$,
where $k\in\N$ satisfies $k > 2/n$
(and $k < 2$ for $n=3$).
Then there is $\omega\sb 1<m$
such that the solitary wave solutions
$\phi\sb\omega e^{-i\omega t}$
to \eqref{nld-nd} described above
are linearly unstable for 
$\omega\in(\omega\sb 1,m)$.
More precisely,
let $A\sb\omega$ be the linearization
of the nonlinear Dirac equation
at a solitary wave $\phi\sb\omega(x)e^{-i\omega t}$.
Then
for $\omega\in(\omega_1,m)$
there are eigenvalues
\[
\pm\lambda\sb\omega\in\sigma\sb{p}(A\sb\omega),
\qquad
\lambda\sb\omega>0,
\qquad
\lambda\sb\omega=O(m-\omega).
\]
\end{theorem}

See Figure~\ref{fig-spectrum}.

\begin{remark}\label{remark-nonlinear}
The existence of an eigenvalue with
positive real part
in the linearization at a particular solitary wave
generally implies the
dynamic, or \emph{nonlinear}, instability
of this wave.
We expect that this could be proved
following the argument of \cite{MR2916078}
given in the context of the nonlinear Schr\"odinger equation.
\end{remark}

\bigskip

\begin{figure}[ht]
\begin{center}
\setlength{\unitlength}{1pt}
\begin{picture}(0,180)(0,-90)
\font\gnuplot=cmr10 at 10pt
\gnuplot

\put(6,58){$i(m+{\omega})$}
\put(5,19){$i(m-{\omega})$}

\put( 23,-11){$\lambda\sb\omega$}
\put(-36,-11){$-\lambda\sb\omega$}

\put(-20,0){\circle*{4}}
\put( 20,0){\circle*{4}}
\put(  0,0){\circle*{4}}

\put(-60,  0){\vector(1,0){120}}
\put(  0,-80){\vector(0,1){160}}

\linethickness{2pt}
\put( 1,-80){\line(0,1){60}}
\put(-1,-80){\line(0,1){20}}
\put( 1,80){\line(0,-1){20}}
\put(-1,80){\line(0,-1){60}}

\end{picture}
\end{center}
\caption{
\small
Main result:
The point spectrum of the linearization
of the nonlinear Dirac equation
in $\R^n$, $n\le 3$,
with
$f(s)=s^k$, $k>\frac{2}{n}$,
at a solitary wave
with $\omega\lesssim m$
contains two nonzero real eigenvalues,
$\pm\lambda\sb\omega$,
with $\lambda\sb\omega=O(m-\omega)$.
See Theorem~\ref{main-theorem-dirac-vbk}.
Also plotted on this picture
is the essential spectrum,
with the edges
at $\lambda=\pm i(m-{\omega})$
and with the embedded threshold points
(branch points of the dispersion relation)
at $\lambda=\pm i(m+{\omega})$.
}

\label{fig-spectrum}

\end{figure}
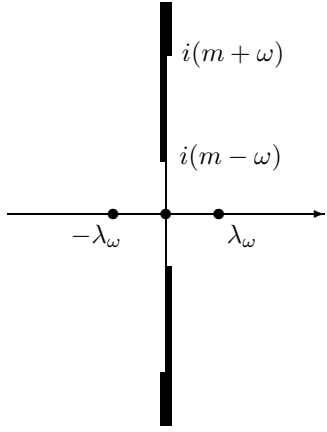

\begin{remark}
Theorem~\ref{main-theorem-dirac-vbk}
extends easily to nonlinearities
$f\in C\sp 2(\R)$
of the form
\[
f(s)=a s^k+O(s^{k+1}),
\qquad
a>0.
\]
\end{remark}

\begin{remark}\label{remark-enough}
In Theorem~\ref{main-theorem-dirac-vbk},
the value $\omega\sb 1<m$
could be taken to be the smallest point such that
there is a $C^1$ family of solitary waves
$\omega\mapsto \phi\sb\omega$
for $\omega\in(\omega\sb 1,m)$
and moreover
$\p\sb\omega Q(\omega)$ does not vanish on $(\omega\sb 1,m)$.
Indeed, by \cite{dirac-vk},
the positive and negative eigenvalues
remain trapped on the real axis,
not being able
to collide at $\lambda=0$
for $\omega\in (\omega\sb 1,m)$
as long as $\p\sb\omega Q(\omega)$
does not vanish on this interval.
These eigenvalues cannot leave into the complex
plane, either, since they are simple,
while the spectrum of the operator
is symmetric with respect to the real
and imaginary axes.
\end{remark}

\begin{remark}
If the family of solitary waves
$\omega\mapsto\phi\sb\omega$
is defined for $\omega\in(\omega\sb 0,m)$
with $\omega\sb 0<\omega\sb 1$
and $\p\sb\omega Q(\omega)$ vanishes at $\omega\sb 1$,
we do not know what happens for $\omega\lesssim\omega\sb 1$.
If $\p\sb\omega Q(\omega)$ changes
the sign at $\omega\sb 1$,
then, generically, either the pair of real
eigenvalues, having collided at $\lambda=0$
when $\omega=\omega\sb 1$,
turn into a pair of purely imaginary eigenvalues
(linear instability disappears),
or instead
two purely imaginary eigenvalues,
having met at $\lambda=0$,
turn into the
second  pair of real eigenvalues
(linear instability persists).
If $\p\sb\omega Q(\omega)$ vanishes at $\omega\sb 1$
but does not change the sign,
then generically the eigenvalues touch
and separate again, remaining on $\R\sb{+}$ and $\R\sb{-}$.
More details are in \cite{dirac-vk}.
\end{remark}

\begin{remark}
Theorem~\ref{main-theorem-dirac-vbk}
is in formal agreement with the Vakhitov--Kolokolov
stability criterion \cite{VaKo},
since
for $\omega\lesssim m$
one has
$Q'(\omega)>0$ for $k>\frac{2}{n}$.
Let us mention that
the sign of the stability criterion,
$Q'(\omega)<0$,
differs from \cite{VaKo}
because of their writing the solitary waves
in the form
$\varphi(x) e^{+i\omega t}$.
\end{remark}

\begin{remark}
It has been shown that in the 1D case with $k=1$,
the small amplitude
solitary waves are spectrally stable
\cite{dirac-spectrum}.
\end{remark}

\begin{remark}
We expect that in the 1D case with $k=2$
(``quintic nonlinearity'')
the small solitary wave solutions
of the nonlinear Dirac equation in 1D
are spectrally stable.
For the corresponding nonlinear Schr\"odinger
equation (quintic nonlinearity in 1D),
the charge is constant,
thus the zero eigenvalue of a linearized operator
is always of higher algebraic multiplicity.
For the Dirac equation, this
degeneracy is ``resolved'':
using the expression for the charge $Q(\omega)$
from \cite[Section 2A]{PhysRevE.82.036604},
one can see that the charge is now a decaying function
for $\omega\lesssim m$
(with nonzero limit as $\omega\to m$),
suggesting that
there are two purely imaginary eigenvalues
$\pm\lambda\sb\omega$
in the spectrum of $A\sb\omega$,
with $\lambda\sb\omega=o(m-\omega)$,
but no eigenvalues with nonzero real part.
\end{remark}

\begin{remark}\label{remark-3d}
Let us notice that
in the 3D case
for the cubic nonlinearity $f(s)=s$
(this is the original Soler model
from \cite{PhysRevD.1.2766}),
based on the numerical evidence from
\cite{PhysRevD.1.2766,PhysRevLett.50.1230},
the charge
$Q(\omega)$
has a local minimum at $\omega\sb 1\approx 0.936 m$,
suggesting that
the solitary waves
with $\omega\sb 1<\omega<1$ are linearly unstable,
but then
at $\omega=\omega\sb 1$ the real eigenvalues
collide at $\lambda=0$,
and there are no nonzero real eigenvalues
in the spectrum
for $\omega\lesssim \omega\sb 1$.
Incidentally, this agrees with the
``dilation-stability'' results of
\cite{MR848095}
(one studies whether the energy is minimized
or not under the charge-preserving
dilation transformations).
\end{remark}

\begin{remark}
We can not rule out the possibility
that the eigenvalues with nonzero real part
could bifurcate directly from the imaginary axis
into the complex plane. Such a mechanism is absent
for the nonlinear Schr\"odinger equation
linearized at a ground state,
for which the point eigenvalues always remain
on the real or imaginary axes.
At present, though, we do not have examples of such
bifurcations
in the context of nonlinear Dirac equation.
\end{remark}

\begin{remark}\label{remark-bl}
For $n=3$, we only consider the case $k=1$,
and we do not consider dimensions $n>3$.
This is because of the fact that
the equation
$-\Delta u+u=\abs{u}^{2k}u$ in $\R^n$
has nontrivial solutions in $H^1(\R^n)$
if and only if $0<k<2/(n-2)$,
as follows from the virial identities;
see \cite{MR0192184}
and \cite[Example 1]{MR695535}.
This is why our method does not allow us
to construct solitary wave solutions
to the nonlinear Dirac equation \eqref{nld-nd}
in $\R^n$, $n\ge 3$, 
with $k\ge 2/(n-2)$.
\end{remark}

\begin{remark}\label{remark-noninteger}
We consider here only integer powers $k$
(and only dimensions $n=1$, $2$, and $3$),
being physically the most important cases.
Mathematically, this merely avoids some minor technical complications
associated with a non-smooth nonlinearity,
and the instability argument can be extended to handle the
corresponding equation with $f(s) = |s|^k$,
in any dimension $n\ge 1$, under the condition $k > 2/n$.
(The restriction $k<2/(n-2)$ is needed
so that there are nontrivial solitary waves in NLS;
see the previous remark.)
\end{remark}

\section{Nonlinear Schr\"odinger and its
solitary waves}
\label{sect-nls}

We are going to use the fact that
the nonrelativistic limit
of the nonlinear Dirac equation
yields the nonlinear Schr\"odinger equation,
\begin{equation}\label{nls}
i\p\sb t\psi=-\frac{1}{2m} \Delta \psi
- |\psi|^{2k} \psi,
\quad
\psi(x,t)\in\C,
\quad
x\in\R^n,
\quad
k>0,
\quad n\in\N.
\end{equation}

\subsection{Solitary waves}

The properties of solitary wave solutions
\[
\psi(x,t)=\phi\sb\omega(x)e^{-i\omega t},
\qquad
\phi\sb\omega\in H^1(\R^n),
\]
with the amplitude $\phi_\omega(x)$
satisfying the stationary equation
\begin{equation}
\label{nls-soleq}
-\frac{1}{2m} \Delta \phi(x) -|\phi|^{2k} \phi
= \omega \phi, \qquad x\in\R^n,
\end{equation}
are well-known \cite{MR0454365,MR695535}.
For any $k>0$ when $n\le 2$
and
for $0<k<2/(n-2)$ when $n\ge 3$,
for each $\omega \in (-\infty,0)$, there
is a unique positive, radially symmetric solution
\[
  \phi_{\omega}(x) = \phi_{\omega}(|x|) > 0,
\]
which decays exponentially. This family
of solitary waves,
known as the {\it ground states},
is generated by rescaling a single amplitude function:
\begin{equation}
\label{nls-scaling}
  \phi_\omega(x) = \abs{\omega}^{\frac{1}{2k}}
  F (\sqrt{2m\abs{\omega}} x),
\qquad
\omega<0,
\end{equation}
where $F(x) = F(|x|) > 0$ solves
\begin{equation}\label{F-such-that}
  -\Delta F - F^{2k+1}
  = -F, \qquad x\in\R^n.
\end{equation}
In one space dimension ($n=1$),
for $k>0$,
$F(x)$ is given by the explicit formula
\[
  F(x) = \left( \frac{k+1}{\cosh^2{k x}}
  \right)^{\frac{1}{2k}}.
\]

\subsection{Linearization at a solitary wave}

To derive the linearization
of the nonlinear Schr\"odinger equation \eqref{nls}
at a solitary wave
$\psi(x,t)=\phi\sb\omega(x)e^{-i\omega t}$,
we use the Ansatz
\[
\psi(x,t)=(\phi\sb\omega(x)+\rho(x,t))e^{-i\omega t},
\qquad
\rho(x,t)\in\C,
\qquad
x\in\R^n,
\]
and arrive at the linearized equation
\begin{equation}\label{nls-lin}
\p\sb t
\bm{\uprho}
=
\eub{j}\eub{l}(\omega)
\bm{\uprho},
\qquad
\bm{\uprho}(x,t)=\begin{bmatrix}\Re\rho(x,t)\\\Im\rho(x,t)\end{bmatrix},
\end{equation}
where
\begin{equation}\label{def-jl}
\eub{j}=
\begin{bmatrix}0&1\\-1&0\end{bmatrix},
\qquad
\eub{l}=
\begin{bmatrix}\eur{l}\sb{-}&0\\0&\eur{l}\sb{+}\end{bmatrix},
\end{equation}
with
$\eur{l}\sb\pm$ self-adjoint Schr\"odinger operators
\begin{equation}\label{def-lpm}
\eur{l}\sb{-}(\omega)=-\frac{1}{2m}\Delta -|\phi\sb\omega|^{2k}-\omega,
\qquad
\eur{l}\sb{+}(\omega)=\eur{l}\sb{-}(\omega) -2k|\phi\sb\omega|^{2k}.
\end{equation}
Since the solitary wave amplitudes $\phi_\omega(x) = \phi_\omega(|x|)$ we take here are
radially symmetric, we may consider the
operators
\begin{equation}\label{def-lr}
  \eub{l}_{rad}, \,\; \eur{l}_{\pm,rad} \; := \;
  \eub{l}, \, \eur{l}_{\pm}  \mbox{ restricted to
  radially symmetric functions}.
\end{equation}

The linear stability theory of NLS ground states
is well understood, and can be summarized,
in terms of their charge \eqref{def-q-omega},
as follows:
\begin{lemma}[Vakhitov--Kolokolov stability criterion \cite{VaKo}]
\label{lemma-vk}
For the linearization
\eqref{nls-lin}
at a ground state solitary wave
$\phi\sb\omega(x)e^{-i\omega t}$,
there are real nonzero eigenvalues
$\pm\lambda\in\sigma\sb{d}(\eub{j}\eub{l})$,
$\lambda>0$,
if and only if
$\frac{d}{d\omega}Q(\omega)>0$
at this value of $\omega$.
If so, then
$\pm\lambda\in\sigma\sb{d}
(\eub{j}\eub{l}_{rad})$ are simple eigenvalues,
and moreover $\ker \eur{l}_{+,rad} = \{ 0 \}$.
\end{lemma}

Using \eqref{nls-scaling},
we compute:
\begin{equation}\label{qp}
Q(\omega)
=\int\sb{\R^n} |\phi\sb\omega(x)|^2\,dx
=\abs{\omega}^{\frac 1 k} \int\sb{\R^n}
F^2(\sqrt{2m\abs{\omega}} x) \,dx
= C \abs{\omega}^{\frac 1 k - \frac n 2},
\quad
\omega<0,
\end{equation}
where
$C=\int\sb{\R^n} F^2(\sqrt{2m}y) \,dy>0$.
We see from \eqref{qp} that
for $\omega<0$
one has
$Q'(\omega)<0$ for $k < 2/n$,
$Q'(\omega)=0$ for $k = 2/n$, and
$Q'(\omega)>0$ for $k > 2/n$.
Thus:

\begin{lemma}\label{lemma-nls-k}
Let $n\in\N$. If $k > n/2$ 
(and $k < 2/(n-2)$ if $n \geq 3$), then
$\sigma\sb{p}(\eub{j}\eub{l}_{rad})\ni\{\pm\lambda\}$,
for some $\lambda>0$,
and in particular the NLS ground states
are linearly unstable.
\end{lemma}

\section{Nonlinear Dirac and its solitary waves}
\label{sect-linearization}

Solitary waves are solutions
to \eqref{nld-nd} of the form
\[
\psi(x,t)=\phi\sb\omega(x)e^{-i\omega t}\sothat
\phi\sb\omega\in H^1(\R,\C^{N}),\ \omega\in\R
\]
and as such the amplitude $\phi_\omega(x)$
must satisfy
\begin{equation}
\label{nld-sol}
  \omega \phi_\omega =-i\bm\alpha\cdot\bm\nabla\psi
+m\beta\psi
-f(\psi\sp\ast\beta\psi)\beta\psi, \qquad x\in\R^n.
\end{equation}

\subsection{Solitary waves in one dimension}
\label{sect-solitary-waves-1d}

We first give a simple demonstration
of the existence and uniqueness of solitary waves in one dimension,
following the article \cite{MR2892774},
and allowing for more general nonlinearities $f(s)$.

\begin{lemma}
\label{lemma-existence-nld-1d}
Let $f(0)=0$.
Denote $g(s)=m-f(s)$,
and let $G(s)$ be the antiderivative of $g(s)$ such that $G(0)=0$.
Assume that
there is $\omega\sb 0<m$
such that
for given $\omega\in(\omega\sb 0,m)$
there exists
$\varGamma\sb\omega>0$ such that
\begin{equation}\label{def-Xi}
\omega\varGamma\sb\omega=G(\varGamma\sb\omega),
\quad
\omega\ne g(\varGamma\sb\omega),
\quad\mbox{and}\quad
\omega s<G(s)\quad{\rm for}
\ s\in(0,\varGamma\sb\omega).
\end{equation}
Then there is a solitary wave solution
$\psi(x,t)=\phi\sb\omega(x)e^{-i\omega t}$
to \eqref{nld-nd},
where
\begin{equation}\label{psi-v-u}
\phi\sb\omega(x)
=\begin{bmatrix}v(x,\omega)\\u(x,\omega)\end{bmatrix},
\end{equation}
with both $v$ and $u$ real-valued,
belonging to $H^1(\R)$ as functions of $x$,
$v$ being even and $u$ odd.

More precisely,
for $x\in\R$
and $\omega\in(\omega\sb 0,m)$,
let us define $\mathscr{X}(x,\omega)$
and $\mathscr{Y}(x,\omega)$ by
\begin{equation}\label{def-xi-eta}
\mathscr{X}=v^2-u^2,\qquad \mathscr{Y}=v u.
\end{equation}
Then $\mathscr{X}(x,\omega)$ is the unique positive symmetric solution to
\begin{equation}
\p\sb x^2\mathscr{X}
=-\p\sb\mathscr{X}(-2G(\mathscr{X})^2+2\omega^2\mathscr{X}^2),
\qquad
\lim\sb{x\to\pm\infty}\mathscr{X}(x,\omega)=0,
\end{equation}
and $\mathscr{Y}(x,\omega)=-\frac{1}{4\omega}\p\sb x\mathscr{X}(x,\omega)$.
This solution satisfies
$\mathscr{X}(0,\omega)=\varGamma\sb\omega$.
\end{lemma}

\begin{proof}
Substituting
$\phi\sb\omega(x)e^{-i\omega t}$,
with $\phi\sb\omega$ from \eqref{psi-v-u},
into (\ref{nld-sol}), we obtain:
\begin{equation}\label{omega-v-u}
\left\{
\begin{array}{ll}
\omega v=\p\sb x u+g(\abs{v}^2-\abs{u}^2)v,
\\
\omega u=-\p\sb x v-g(\abs{v}^2-\abs{u}^2)u.
\end{array}
\right.
\end{equation}
Since we assume that both $v$ and $u$ are real-valued,
we may rewrite (\ref{omega-v-u}) as the following
Hamiltonian system:
\begin{equation}\label{stat-eqn}
\left\{
\begin{array}{ll}
\p\sb x u=\omega v-g(v^2-u^2)v=\p\sb{v} h(v,u),
\\
-\p\sb x v=\omega u+g(v^2-u^2)u=\p\sb{u} h(v,u),
\end{array}
\right.
\end{equation}
where the Hamiltonian $h(v,u)$ is given by
\begin{equation}\label{def-h}
h(v,u)=\frac{\omega}{2}(v^2+u^2)-\frac{1}{2}G(v^2-u^2).
\end{equation}
The solitary wave
with a particular $\omega\in(\omega\sb 0,m)$
corresponds to a trajectory of this Hamiltonian
system such that
\[
\lim\sb{x\to\pm\infty}v(x,\omega)
=\lim\sb{x\to\pm\infty}u(x,\omega)=0,
\]
hence
$\lim\sb{x\to\pm\infty}\mathscr{X}=0$.
Since $G(s)$ satisfies
$G(0)=0$,
we conclude that
$
h(v(x),u(x))\equiv 0,
$
which leads to
\begin{equation}\label{phi-phi-g}
\omega(v^2+u^2)=G(v^2-u^2).
\end{equation}
We conclude from
\eqref{phi-phi-g}
that
solitary waves may only correspond to $\abs{\omega}<m$,
$\omega\ne 0$.


The functions $\mathscr{X}(x,\omega)$ and $\mathscr{Y}(x,\omega)$
introduced in
\eqref{def-xi-eta} are to solve
\begin{equation}\label{xi-eta-system}
\left\{
\begin{array}{ll}
\p\sb x\mathscr{X}=-4\omega\mathscr{Y},
\\
\p\sb x\mathscr{Y}=-(v^2+u^2)g(\mathscr{X})+\omega\mathscr{X}
=-\frac{1}{\omega}G(\mathscr{X})g(\mathscr{X})+\omega\mathscr{X},
\end{array}
\right.
\end{equation}
and to have the asymptotic behavior
$\lim\sb{\abs{x}\to\infty}\mathscr{X}(x)=0$,
$\lim\sb{\abs{x}\to\infty}\mathscr{Y}(x)=0$.
In the second equation in
(\ref{xi-eta-system}), we used the relation (\ref{phi-phi-g}).
The system (\ref{xi-eta-system}) can be
written as the following equation on $\mathscr{X}$:
\begin{equation}\label{xi-p-p}
\p\sb x^2\mathscr{X}
=-\p\sb\mathscr{X}(-2 G(\mathscr{X})^2+2\omega^2\mathscr{X}^2)
=4
\big(G(\mathscr{X})g(\mathscr{X})-\omega^2\mathscr{X}\big).
\end{equation}
This equation describes a particle in the potential
$-2 G(s)^2+2\omega^2 s^2$.
The condition \eqref{def-Xi}
is needed so that $s=\varGamma\sb\omega$
is the turning point
for the zero energy trajectory in this potential.
The existence of a positive solution
$\mathscr{X}(x,\omega)$ follows.
This solution is unique up to a translation,
and it will be made symmetric in $x$
by requiring $\mathscr{X}(0,\omega)=\varGamma\sb\omega$.
\end{proof}

\subsection{Solitary waves in the nonrelativistic limit}
\label{sect-sw-hd}

In dimensions $n=1$, $2$ and $3$, we consider solitary wave amplitudes $\phi_\omega(x)$ of the
forms given in~\eqref{sol-forms-1} and \eqref{sol-forms-23}.
Substituting these into the nonlinear
Dirac equation~\eqref{nld-nd}, a
straightforward calculation results in
the system
\begin{equation}
\label{soleq-nd}
\left\{ \begin{array}{l}
\omega v = \p_r u + \frac{n-1}{r} u +
m v-f(v^2-u^2)v \\
\omega u = -\p_r v - m u+f(v^2-u^2)u
\end{array} \right.
\end{equation}
for the pair of real-valued functions $v = v(r,\omega)$,
$u = u(r,\omega)$. Notice that
equation~\eqref{soleq-nd} includes the
$1$-dimensional case~\eqref{omega-v-u}
if we interpret $r=x$.

Recalling that $f(s)=s^{k}$, we arrive at
\[
(\omega-m) v
= \p_r u + \frac{n-1}{r} u - f v,
\qquad
(\omega+m) u
= - \p_r v + f u,
\]
with
\[
  f:=(v^2-u^2)^{k}.
\]

To consider the nonrelativistic limit, we set
\[
  m^2 - \omega^2 = \epsilon^2, \quad
  0 < \epsilon \ll m,
\]
and rescale $v(r,\omega)$ and $u(r,\omega)$ as follows:
\[
  v(r,\omega) =\epsilon^{\frac 1 k}
  V(\epsilon r,\epsilon), \qquad
  u(r,\omega) =\epsilon^{1+\frac 1 k}
  U(\epsilon r,\epsilon).
\]
Then $V$, $U$ should satisfy
\[
\left\{ \begin{array}{l}
(\omega-m)\epsilon^{\frac 1 k} V
=\epsilon^{2+\frac 1 k} (\p\sb RU + \frac{n-1}{R} U)
- \epsilon^{1/k} f V \\
(\omega+m)\epsilon^{1+\frac 1 k} U
=-\epsilon^{1+\frac 1 k} \p\sb R V
+ \epsilon^{1 + \frac 1 k} f U
\end{array} \right.,
\]
where $R = \epsilon r$ denotes the ``rescaled variable".
Using $\omega = m - \frac{1}{2m} \epsilon^2
+ O(\epsilon^4)$,
and taking into account that
\[
f = (\epsilon^{\frac{2}{k}} V^2
- \epsilon^{2+\frac{2}{k}} U^2 )^{k}
= \epsilon^2 V^{2k} + \epsilon^4
O(U^{2k} + V^{2k}),
\]
we re-write the system as
\begin{equation}\label{zero-is-phi1}
\left\{ \begin{array}{l}
(-\frac{1}{2m} + O(\epsilon^2)) V =
\p\sb R U + \frac{n-1}{R} U - V^{2k+1}
+ \epsilon^2 O((U^{2k} + V^{2k})|V|) \\
(2m + O(\epsilon^2)) U = -\p\sb R V
+ \epsilon^2 V^{2k} U +
\epsilon^4 O((U^{2k} + V^{2k})|U|)
\end{array} \right..
\end{equation}
The rescaled system~\eqref{zero-is-phi1}
has an obvious limit as $\epsilon\to 0$.
Formally (for now)
setting
\[
\hat V(r)=\lim\sb{\epsilon\to 0}V(r,\epsilon),
\qquad
\hat U(r)=\lim\sb{\epsilon\to 0}U(r,\epsilon),
\]
we arrive at
\begin{equation}\label{def-hat-phi}
-\frac{1}{2m} \hat V = \p\sb R \hat U + \frac{n-1}{R}
\hat U - \hat V^{2k+1}, \qquad
2m \hat U = -\p\sb R\hat V.
\end{equation}
Substituting the second equation into the first one
yields
\[
  -\frac{1}{2m} (\p\sb R^2+\frac{n-1}{R}\p\sb R)\hat V - \hat V^{2k+1}
  = -\frac{1}{2m} \hat V, \qquad
  \hat{U} = - \frac{1}{2m} \p\sb R \hat V.
\]
This equation for $\hat V(r)$ is precisely
the equation~\eqref{nls-soleq} for
NLS solitary wave amplitudes $\phi_\omega$
with $\omega = -\frac{1}{2m}$.
Thus we let
$\hat V(r)$ be the (unique) NLS ground state:
\begin{equation}
\label{Vhatdef}
  \hat V(r) := (2m)^{-\frac{1}{2k}}
  F(r), \qquad
  \hat U(r) :=
-(2m)^{-\frac{1}{2k}-1}
  F'(r),
\end{equation}
with $F(r)$ the unique positive
spherically symmetric solution to \eqref{F-such-that}.
We can use this nonrelativistic limit to
construct nonlinear Dirac solitary waves for
$\epsilon^2 = m^2 - \omega^2 \ll m^2$:
\begin{lemma}\label{lemma-xx-ff-3}
There is $\omega\sb 0 < m$
such that for $\omega = \sqrt{m^2-\epsilon^2} \in (\omega\sb 0,m)$,
there are solutions of~\eqref{soleq-nd} of the form
\[
v(r,\omega) = \epsilon^{\frac 1 k} \left[
\hat V(\epsilon r) + \tilde V(\epsilon r) \right],
\quad
u(r,\omega)=\epsilon^{1+\frac 1 k} \left[
\hat U (\epsilon r) + \tilde U(\epsilon r) \right],
\]
\[
\| \tilde V \|_{H^2} + \| \tilde U \|_{H^2}
= O(\epsilon^2).
\]
\end{lemma}
\begin{remark}
In the one-dimensional case, since the
solitary waves are unique (up to symmetries),
it follows that these asymptotics describe {\it every} solitary wave for $\omega$ close to $m$,
or equivalently every small amplitude solitary wave.
\end{remark}
\begin{proof}
The argument parallels that of~\cite{2008arXiv0812.2273G}, where
the (more general) nonlinearity
$f(s) = \abs{s}^{\theta}$, $0 < \theta < 2$
is considered for $n=3$. Writing
\[
V(R,\epsilon) = \hat{V}(R) + \tilde{V}(R,\epsilon),
\qquad
U(R,\epsilon) = \hat{U}(R) + \tilde{U}(R,\epsilon),
\]
and subtracting equations~\eqref{zero-is-phi1}
and \eqref{def-hat-phi}, we arrive at
\[
\begin{split}
&-\frac{1}{2m}\tilde V + O(\epsilon^2) V
=
(\p_R + \frac{n-1}{R}) \tilde{U} -
(2k+1) \hat V^{2k} \tilde{V}
\\
&\qquad\qquad
+ O(|\hat V|^{2k+1}
\tilde V^2 + |\tilde V|^{2k+1})
+ \epsilon^2 O((U^{2k}+V^{2k})|V|),
\end{split}
\]
\[
2m \tilde U + O(\epsilon^2) U =
- \p\sb R\tilde V + \epsilon^2 V^{2k} U
+ \epsilon^4 O((U^{2k}+V^{2k})|U|),
\]
which, setting
\[
  \Xi(R,\epsilon) := \left[ \begin{array}{c}
  \tilde V (R,\epsilon) \\ \tilde U (R,\epsilon)
  \end{array} \right],
\]
we may re-write as
\[
\label{HH}
\eub{H} \Xi = O_{H^1} (\epsilon^2)
+ O(\epsilon^2 |\Xi| + |\Xi|^2 + |\Xi|^{2k+1}),
\]
where
\begin{equation}\label{def-h-large}
\eub{H} :=
\left[
\begin{array}{cc}
-\frac{1}{2m} + (2k+1) \hat V^{2k}
&\ \ - (\p_R + \frac{n-1}{R}) \\
\p_R & 2m
\end{array} \right].
\end{equation}

Since
\[
\left[ \begin{array}{c} \xi \\ \eta
\end{array} \right] \in \ker \eub{H} \;\;
\iff \;\; \xi \in \ker\eur{l}_+, \;\;
\eta = -\frac{1}{2m} \p\sb R \xi,
\]
and $\ker\eur{l}_{+,rad} = \{  0 \}$
(Cf. definitions \eqref{def-lpm}, \eqref{def-lr}),
we see that $\ker \eub{H} = \{ 0 \}$.
It then follows from the fact that
$\eur{l}_{+,rad}^{-1}$ is bounded from
$L^2_r(\R^n,\C)$ to $H^2(\R^n,\C)$,
that  $\eub{H}^{-1}$ is bounded from
$H^1_r(\R^n,\C^2)$
to $H^2(\R^n,\C^2)$
(here
$L^2_r$, $H^1_r$
are the corresponding subspaces of spherically symmetric functions).
Hence
\begin{equation}
\label{Xi}
\Xi = \eub{H}^{-1} \left\{
O_{H^1} (\epsilon^2)
+ O(\epsilon^2 |\Xi| + |\Xi|^2 + |\Xi|^{2k+1})
\right\} ,
\end{equation}
and since
$\| \Xi \|_{L^\infty} \leq C \| \Xi \|_{H^2}$
(recall that $n \leq 3$), we arrive easily at
\[
\begin{split}
  \| R.H.S. \eqref{Xi} \|_{H^2} &\leq C
  \| O_{H^1} (\epsilon^2)
  + O(\epsilon^2 |\Xi| + |\Xi|^2 + |\Xi|^{2k+1} )
  \|_{H^1} \\ &\leq C
  \left\{ \epsilon^2 + \epsilon \|\Xi\|_{H^2}
  + \| \Xi \|_{H^2}^2 + \| \Xi \|_{H^2}^{2k+1}
  \right\},
\end{split}
\]
and so we see that for small enough $\epsilon$,
the map on the r.h.s. of~\eqref{Xi}
maps the ball of radius $\epsilon$ in $H^2$
into itself. A similar estimate shows that
this map is a contraction, and hence has a
unique fixed point $\Xi$ in this ball. Finally,
we see from~\eqref{Xi} that
$\| \Xi \|_{H^2} = O(\epsilon^2)$.
\end{proof}

\section{Linear instability of small amplitude solitary waves}
\label{sect-instability}

Our first observation here is that on spinor fields of the form
\[
\psi(x,t) =\begin{bmatrix} \Psi_1(x,t) \\ \Psi_2(x,t) \end{bmatrix},
\qquad
\psi(x,t) =\begin{bmatrix} \Psi_1(r,t)\\
i e^{i\phi} \Psi_2(r,t)\end{bmatrix},
\]
\[
\psi(x,t) =
\begin{bmatrix}
\Psi_1(r,t)
\begin{pmatrix}
1\\0
\end{pmatrix}
\\
i \Psi_2(r,t)
\begin{pmatrix}
\cos\theta
\\
e^{i\phi}\sin\theta
\end{pmatrix}
\end{bmatrix}
\]
in dimensions $n=1$, $n=2$, and $n=3$
respectively, the nonlinear Dirac
equation~\eqref{nld-nd} reduces to the system
\begin{equation}
\label{reduced}
\left\{ \begin{array}{l}
i \p_t \Psi_1 = (\p_r + \frac{n-1}{r}) \Psi_2
+ m \Psi_1-f(|\Psi_1|^2-|\Psi_2|^2) \Psi_1 \; , \\
i \p_t \Psi_2 = -\p_r \Psi_1 - m\Psi_2+f(|\Psi_1|^2-|\Psi_2|^2) \Psi_2
\end{array} \right.
\end{equation}
(with the convention $r=x$ for $n=1$).
The solitary waves considered in
\eqref{sol-forms-1} and \eqref{sol-forms-23}
lie in this
class of fields, corresponding to
\begin{equation}\label{sol-forms-all}
\Psi_1(r,t)=v(r,\omega)e^{-i\omega t},
\qquad
\Psi_2(r,t)=u(r,\omega)e^{-i\omega t}.
\end{equation}
To prove the instability of these solitary waves,
it suffices to show that they are unstable as
solutions of~\eqref{reduced}.

\subsection{Linearization at a solitary wave}

To derive the linearization of system~\eqref{reduced}
at a solitary wave
\eqref{sol-forms-all}
we consider solutions in the form of the Ansatz
\begin{equation}
\left[ \begin{array}{c} \Psi_1(r,t)  \\ \Psi_2(r,t)
\end{array} \right] =
\left( \left[ \begin{array}{c} v(r,\omega) \\
u(r,\omega) \end{array} \right]  +
\left[ \begin{array}{c} \rho_1(r,t) \\
\rho_2(r,t) \end{array} \right] \right)
e^{-i\omega t},
\end{equation}
where
\[
\rho(r,t) := \left[ \begin{array}{c} \rho_1(r,t) \\
\rho_2(r,t) \end{array} \right] \in \C^2.
\]
Inserting this Ansatz into system~\eqref{reduced}
with $f(s)=s^k$ and recalling that
$u(r,\omega)$ and $v(r,\omega)$ are real-valued,
we find that the linearized system for $\rho$ is
\begin{eqnarray}
\label{nld-1d-lin}
i \p_t \rho =
\begin{bmatrix}
m- \omega - (v^2-u^2)^k & \p_r + \frac{n-1}{r} \\
-\p_r & -m-\omega+ (v^2-u^2)^k
\end{bmatrix}
\rho
\nonumber
\\
\qquad
-
2k(v^2-u^2)^{k-1}
\begin{bmatrix}
v^2
&-u v
\\
-u v
&u^2
\end{bmatrix}
\Re\rho.
\end{eqnarray}
We note that the above equation is $\R$-linear
but not $\C$-linear, due to the presence of
the term with
$\Re\rho=\begin{bmatrix}\Re\rho_1\\\Re\rho_2\end{bmatrix}$.
We rewrite equation~\eqref{nld-1d-lin}
in terms of
$\Re\rho\in\R^2$
and
$\Im\rho\in\R^2$:
\begin{equation}\label{nld-1d-lin-2}
\p\sb t
\begin{bmatrix}\Re\rho\\\Im\rho\end{bmatrix}
= \eubJ \eubL(\omega)
\begin{bmatrix}\Re\rho\\\Im\rho \end{bmatrix}
\end{equation}
where $\eubJ$ corresponds to $1/i$:
\[
\eubJ=
\begin{bmatrix}
0&I\sb 2\\-I\sb 2&0
\end{bmatrix},
\]
and the $4 \times 4$ matrix operator $\eubL(\omega)$ is defined by
\[
\eubL(\omega) =
\begin{bmatrix} \eurL_{+}(\omega)  & 0 \\ 0 & \eurL_{-}(\omega) \end{bmatrix} ,
\]
where, writing
\[
  f := f(\phi_\omega^* \beta \phi_\omega)
  = (v^2(r,\omega) - u^2(r,\omega))^k, \quad
  f' := f'(\phi_\omega^* \beta \phi_\omega),
\]
we have
\begin{equation}
\label{defD-}
\eurL_{-}(\omega) =
\begin{bmatrix}
m-\omega-f& \p_r + \frac{n-1}{r} \\
-\p_r & -m-\omega+f
\end{bmatrix}
\end{equation}
and
\begin{equation}
\label{defD+}
\begin{split}
&\eurL_{+}(\omega) = \eurL_{-}(\omega)
-2 f' \begin{bmatrix}
v^2 & -u v \\ -u v & u^2
\end{bmatrix}.
\end{split}
\end{equation}
Let us remind the reader that $v = v(r,\omega)$
and $u = u(r,\omega)$
in~\eqref{defD-}-\eqref{defD+}
both depend on $\omega$.

Then equation \eqref{nld-1d-lin-2}
which describes the linearization of the
reduced system~\eqref{reduced} at the solitary wave $\phi\sb\omega e^{-i\omega t}$,
takes the form
\[
\p\sb t
\begin{bmatrix}\Re\rho\\\Im\rho\end{bmatrix}
=
\eubJ\eubL(\omega)
\begin{bmatrix}\Re\rho\\\Im\rho\end{bmatrix}
=
\begin{bmatrix}0&\eurL\sb{-}(\omega)\\-\eurL\sb{+}(\omega)&0\end{bmatrix}
\begin{bmatrix}\Re\rho\\\Im\rho\end{bmatrix}.
\]

For the sake of completeness we record
here the essential spectrum of the linearized operator:

\begin{lemma}\label{lemma-cont}
$
\sigma\sb{\!\rm ess}(\eubJ\eubL(\omega))
=
i\R\backslash i(\abs{\omega}-m,m-\abs{\omega}).
$
\end{lemma}

\begin{proof}
The proof follows from noticing that,
due to the exponential spatial decay of
$v(r,\omega)$, $u(r,\omega)$
and the Weyl theorem on the essential spectrum
\cite[Theorem XIII.14, Corollary 2]{MR0493421},
which leads to
\[
\sigma\sb{\!\rm ess}(\eubJ\eubL(\omega))
=\sigma\sb{\!\rm ess}(\eubJ(\eubD\sb m-\omega)),
\]
where
\[
\eubD\sb m=\begin{bmatrix}D\sb m&0\\0&D\sb m\end{bmatrix},
\qquad
D\sb m=i\sigma\sb 2\p\sb r+\sigma\sb 3 m.
\]
At the same time,
since
$D\sb m^2=-\Delta+m^2$,
$\sigma\sb{\!\rm ess}(D\sb m)=\R\backslash(-m,m)$,
while
$\eubJ$ commutes with $\eubD\sb m$ and
$\sigma(\eubJ)=\{\pm i\}$,
one concludes that
\[
\sigma\sb{\!\rm ess}(\eubJ(\eubD\sb m-\omega))
=
\sigma\sb{\!\rm ess}(i(D\sb m-\omega))
\cup
\sigma\sb{\!\rm ess}(-i(D\sb m-\omega))
=i\R\backslash i(\abs{\omega}-m,m-\abs{\omega}).
\]
\end{proof}

\subsection{Unstable eigenvalue of $\eubJ\eubL$ for $\omega\lesssim m$}
\label{sect-5.2}

\begin{proposition}\label{prop-k3}
Let $k\in\N$ satisfy $k>2/n$.
There is $\omega\sb 1<m$
(which depends on $n$ and $k$)
such that
for $\omega\in(\omega\sb 1,m)$
there are two families of eigenvalues
\[
\pm\lambda\sb\omega\in\sigma\sb p(\eubJ\eubL(\omega)),
\qquad
\mbox{with}
\quad
\lambda\sb\omega>0,
\quad\lambda\sb\omega=O(m-\omega).
\]
\end{proposition}

\begin{proof}
The relation
$
\begin{bmatrix}0&\eurL\sb{-}\\-\eurL\sb{+}&0\end{bmatrix}
\begin{bmatrix}\varphi\\\vartheta\end{bmatrix}
=
\lambda
\begin{bmatrix}\varphi\\\vartheta\end{bmatrix}
$,
with
$\varphi,\,\vartheta\in\C^2$,
can be written explicitly as follows:
\begin{equation}
\begin{bmatrix}
-\lambda & 0 & m-\omega-f & \p_r + \frac{n-1}{r}
\\
0 & -\lambda & -\p_r & f-m-\omega
\\
\omega-m+f+2f'v^2 &
-\p_r-2f'v u & -\lambda & 0
\\
\p_r-2f'v u & m+\omega-f +2 f'u^2 & 0 & -\lambda
\end{bmatrix}
\begin{bmatrix}\varphi\sb 1\\\varphi\sb 2\\\vartheta\sb 1\\\vartheta\sb 2\end{bmatrix}
=0.
\end{equation}
We divide the first and the third rows by $\epsilon^2= m^2 - \omega^2$,
the second and the fourth rows by $\epsilon$,
and substitute $R=\epsilon r$,
$\varphi\sb 2=\epsilon\varPhi\sb 2$,
$\vartheta\sb 2=\epsilon\varTheta\sb 2$,
to get
\begin{equation}\label{lhpm4}
\begin{bmatrix}
-\frac{\lambda}{\epsilon^2} &0&\frac{m-\omega-f}{\epsilon^2} & \p_R + \frac{n-1}{R}
\\
0 & -\lambda & -\p_R &-m-\omega
\\
\frac{\omega-m+f+2f'v^2}{\epsilon^2}
&-(\p_R + \frac{n-1}{R}) -\frac{2f'v u}{\epsilon}
&-\frac{\lambda}{\epsilon^2} & 0
\\
\p_R-\frac{1}{\epsilon}2f'v u &
m+\omega-f +2f'u^2
& 0 & -\lambda
\end{bmatrix}
\begin{bmatrix}\varphi\sb 1\\ \varPhi\sb 2\\\vartheta\sb 1\\
\varTheta\sb 2\end{bmatrix}
=0.
\end{equation}
Anticipating the $\epsilon \to 0$ limit, formally
set $\Lambda=\lim\limits\sb{\epsilon\to 0}\frac{\lambda}{\epsilon^2}$,
and introduce the matrices
\begin{equation}\label{def-akk}
\eub{A}\sb\Lambda=\begin{bmatrix}
-\Lambda
&0& \frac{1}{2m}-\hat{V}^{2k}(R) &
\p_R + \frac{n-1}{R}
\\
0&0&-\p_R&-2m
\\
-\frac{1}{2m} +(2k+1)\hat{V}^{2k}(R)
&-(\p_R + \frac{n-1}{R}) &-\Lambda &0
\\
\p_R &2m &0&0
\end{bmatrix},
\end{equation}
\begin{equation}\label{def-kk}
\eub{K}\sb 1=\mathop{\rm diag}[1,0,1,0],
\qquad
\eub{K}\sb 2=\mathop{\rm diag}[0,1,0,1],
\end{equation}
where $\hat{V}(R)$, the NLS ground state, was introduced in \eqref{Vhatdef}.
We write
\eqref{lhpm4} in the form
\begin{equation}\label{lhpm5}
\eub{A}\sb\Lambda
\eta =
\Big(\frac{\lambda}{\epsilon^2}-\Lambda\Big)\eub{K}\sb 1\eta
+\lambda \eub{K}\sb 2\eta
+W\eta,
\qquad
\eta=\begin{bmatrix}\varphi\sb 1\\\varPhi\sb 2\\\vartheta\sb 1\\\varTheta\sb 2\end{bmatrix}
\in\C^{4},
\end{equation}
where
$W(R,\epsilon)$
is a zero order differential operator with $L\sp\infty$ coefficients.

\begin{lemma}\label{lemma-bounds-w}
$
\norm{W(\cdot,\epsilon)}\sb{L\sp\infty(\R^+,\C^{4}\to\C^{4})}
\le O(\epsilon^{2}).
$
\end{lemma}
\begin{proof}
By Lemma~\ref{lemma-xx-ff-3}
(and the Sobolev inequality), one has
\[
v(r,\omega)
=\epsilon^{\frac 1 k}V(\epsilon r)
=\epsilon^{\frac 1 k}
(\hat{V}(\epsilon r) + \tilde{V}(\epsilon r,\epsilon)),
\]
\[
u(r,\omega)
=
\epsilon^{1+\frac 1 k}U(\epsilon r)
=\epsilon^{1+\frac 1 k}
(\hat{U}(\epsilon r) + \tilde{U}(\epsilon r,\epsilon)),
\]
\[
\| \tilde{V} \|_{L^\infty} +
\| \tilde{U} \|_{L^\infty} = O(\epsilon^2).
\]
Then
\[
f(v^2-u^2) =\epsilon^{2}(V^2 - U^2)^k,
\]
\[
f'(v^2-u^2) = k \epsilon^{2 - \frac 2 k}(V^2 - \epsilon^2 U^2)^{k-1},
\]
\[
\begin{split}
-\frac{1}{2m}
+\frac{m - \omega-f - 2 f' v^2}{\epsilon ^2}
&
=
O(\epsilon^2) - (V^2-\epsilon^2 U^2)^k - 2k V^2(V^2-\epsilon^2 U^2)^{k-1}
\\ &
=
- (1+2k)\hat{V}^{2k}
+ O_{L^\infty}(\epsilon^2),
\end{split}
\]
\[
\frac{f' v u}{\epsilon} =
k \epsilon^2 UV(V^2-\epsilon U^2)^{k-1}
= O_{L^\infty}(\epsilon^2),
\]
\[
m-\omega+f = O(\epsilon^2) +
\epsilon^2(V^2-\epsilon^2 U^2)^k
= O_{L^\infty}(\epsilon^2),
\]
\[
f' u^2 = k \epsilon^4 U^2 (V^2 - \epsilon U^2)^{k-1}
= O_{L^\infty}(\epsilon^4),
\]
and
\[
\frac{1}{2m}
+
\frac{\omega-m+f}{\epsilon^2}
=
+ O(\epsilon^2) +
(V^2 - \epsilon^2 U^2)^k
=
\hat{V}^{2k} + O_{L^\infty}
(\epsilon^2),
\]
and the Lemma follows directly from this list of
estimates.
\end{proof}

\begin{lemma}\label{lemma-ker-a}
$\dim\ker \eub{A}\sb\Lambda =
\dim \ker (\eub{j} \eub{l}_{rad} - \Lambda)$,
where
\[
\eub{j}\eub{l}_{rad}=\begin{bmatrix}
0 & \eur{l}_{+,rad} \\ -\eur{l}_{-,rad} & 0 \end{bmatrix},
\]
and where, we recall,
\[
\eur{l}_{-,rad} = -\frac{1}{2m}(\p_R + \frac{n-1}{R})
\p_R +\frac{1}{2m}-\hat{V}(R)^{2k},
\]
\[
\eur{l}_{+,rad} = -\frac{1}{2m} (\p_R + \frac{n-1}{R})
\p_R  + \frac{1}{2m} - (2k+1)\hat{V}(R)^{2k}.
\]
Moreover, if $k > n/2$
($k=1$ if $n=3$),
there is $\Lambda>0$ such that
$\pm\Lambda\in\sigma\sb{d}(\eub{j}\eub{H})$
are simple eigenvalues.
Here
$\eub{H}$ is the operator defined in \eqref{def-h-large}.
\end{lemma}

\begin{proof}
An easy computation shows that
$
  \Phi = \left[ \begin{array}{c}
  \Phi_1 \\ \Phi_2 \\ \Phi_3 \\ \Phi_4
  \end{array} \right] \in \ker \eub{A}_\Lambda
$
if and only if
\[
  2m\Phi_2 = -\p\sb R\Phi_1,
\quad
  2m\Phi_4 = -\p\sb R\Phi_3,
\quad
  (\eub{j} \eub{l}_{rad} - \Lambda)
 \left[ \begin{array}{c} \Phi_3 \\ -\Phi_1 \end{array}
  \right] = 0,
\]
and the first statement of the Lemma
follows from this observation.
The second statement then follows from
Lemma~\ref{lemma-nls-k}.
\end{proof}

So we may assume that there is $\Lambda>0$
such that $\pm \Lambda \in \sigma_d (\eub{A}_\Lambda)$, with eigenfunctions
\[
  \ker \eub{A}_{\pm \Lambda} \ni
  \Phi_{\pm \Lambda} = \left[  \begin{array}{c}
  \pm \Phi_{1} \\ \mp \frac{1}{2m} \p\sb R\Phi_1 \\
  \Phi_3 \\  -\frac{1}{2m}\p\sb R\Phi_{3}
  \end{array} \right], \;\;
  \eur{l}_{+,rad} \Phi_1 = -\Lambda \Phi_3, \;\;
  \eur{l}_{-,rad} \Phi_3  = \Lambda \Phi_1.
\]
We will use the Rayleigh--Schr\"odinger perturbation theory
to show that there are eigenvalues
$\pm\lambda\in\sigma\sb{d}(\eubJ\eubL)$
with $\lambda=\epsilon^2\Lambda+o(\epsilon^2)$.

Writing
\[
  \eub{A}_\Lambda = \eubJ \eubL_0  -
  \Lambda \eub{K}_1, \quad
  \eubL_0 = \begin{bmatrix}
  \eurL_{0,+} & 0 \\ 0 & \eurL_{0,-} \end{bmatrix},
\]
with
\[
  \eurL_{0,+} = \begin{bmatrix}
  \frac{1}{2m} - (2k+1) \hat{V}^{2k} & \p_R + \frac{n-1}{R} \\ -\p_R & -2m \end{bmatrix},
  \quad \eurL_{0,-} = \begin{bmatrix}
  \frac{1}{2m} - \hat{V}^{2k} & \p_R + \frac{n-1}{R}
  \\ -\p_R & -2m \end{bmatrix}
\]
self-adjoint, we see that
\[
  \eub{A}_\Lambda^* = - \eubL_0 \eubJ
  - \Lambda \eub{K}_1 = \eub{F} \eub{A}_\Lambda \eub{F}, \quad
  \eub{F} := \begin{bmatrix} 0 & I_2 \\
  I_2 & 0 \end{bmatrix}.
\]
Hence $\ker \eub{A}^*_\Lambda$ is spanned by $\Phi_\Lambda^*  := \eub{F} \Phi_\Lambda$.
Let $\eub{P}_\Lambda$
denote the orthogonal projection onto $\Phi_\Lambda^*$.


Seeking $\eta$ and $\lambda$ in the form
\[
  \eta=\Phi\sb\Lambda+\zeta, \quad
  \zeta \perp \Phi_\Lambda, \qquad
  \lambda = \epsilon^2(\Lambda + \mu),
\]
then \eqref{lhpm5} becomes
\begin{equation}
\label{eveq}
  \eub{A}_\Lambda \zeta =
  \mu \eub{K}_1 ( \Phi_\Lambda + \zeta ) +
  \epsilon^2 (\Lambda + \mu) \eub{K}_2
  (\Phi_\Lambda + \zeta) + W (\Phi_\Lambda
  + \zeta).
\end{equation}

Applying $\eub{P}_\Lambda$ and $1- \eub{P}_\Lambda$ to \eqref{eveq}, one has:
\begin{equation}\label{eq1}
0= \mu \langle \Phi\sb\Lambda^*,\eub{K}\sb 1(\Phi_\Lambda + \zeta)\rangle
+ \epsilon^2(\Lambda + \mu) \langle \Phi\sb\Lambda^*,\eub{K}\sb 2(\Phi_\Lambda+\zeta)\rangle
+\langle \Phi\sb\Lambda^*,W(\Phi+\zeta)\eta\rangle,
\end{equation}
\begin{equation}\label{eq2}
\eub{A}_\Lambda \zeta= (1-\eub{P}\sb\Lambda)
\Big( \mu \eub{K}\sb 1
+ \epsilon^2(\Lambda + \mu) \eub{K}\sb 2
+W \Big)(\Phi_\Lambda+\zeta).
\end{equation}

\begin{lemma}\label{lemma-nondegen}
\[
\langle \Phi\sb\Lambda^*, \eub{K}\sb 1\Phi\sb\Lambda \rangle \ne 0.
\]
\end{lemma}

\begin{proof}
Note that
\[
  \langle \Phi_\Lambda^*, \eub{K}_1 \Phi_\Lambda \rangle=
  \langle \eub{F} \Phi_\Lambda, \eub{K}_1
  \Phi_\Lambda \rangle =
  2 \Re \langle \Phi_3, \Phi_1 \rangle.
\]
Now the fact that
$\Phi_\Lambda \in \ker \eub{A}_\Lambda$
means in particular that
$L_{-} \Phi_3 = \Lambda \Phi_1$.
Hence, since $L_{-}$ is self-adjoint,
\[
  \R \ni \langle \Phi_3, L_{-} \Phi_3 \rangle
  = \Lambda \langle \Phi_3, \Phi_1 \rangle,
\]
and so
$\Re \langle \Phi_3, \Phi_1 \rangle= 0$
 only if $\langle \Phi_3, L_{-} \Phi_3 \rangle =0$. As is well-known, since
$L_{-} \hat{V} = 0$ and $\hat{V}(r) > 0$,
we have $L_{-} \geq 0$. Thus
$\Re \langle \Phi_3, \Phi_1 \rangle= 0$
only if $L_{-} \Phi_3 = 0$. This, in turn,
would imply that either $\Lambda = 0$
or $\Phi_\Lambda = 0$, both of which
are false.
This finishes the proof.
\end{proof}

Denote by $L^2_r(\R^n,\C^4)\subset L^2(\R^n,\C^4)$
the subspace of spherically symmetric functions.
Now using~Lemma~\ref{lemma-nondegen}
and
the existence of the bounded inverse $\eub{A}_\Lambda^{-1} : \range (1 - \eub{P}_\Lambda) \to \Phi_\Lambda^{\perp}$,
equations \eqref{eq1}, \eqref{eq2} can be written as
\[
\mu=M(\mu,\zeta),
\qquad
\zeta=Z(\mu,\zeta),
\]
with functions
$
M:\R\times L^2_r \to\R, \;\;
$
$
Z:\R\times L^2_r \to L^2_r
$
given by
\begin{equation}\label{def-map-m}
M(\mu,\zeta)=
-\frac{1}
{\langle \Phi\sb\Lambda^*,\eub{K}\sb 1\Phi\sb\Lambda\rangle}
\Big[
\mu
\langle \Phi\sb\Lambda^*,\eub{K}\sb 1\zeta\rangle
+
\langle \Phi_\Lambda^*,
\epsilon^2
(\Lambda+\mu)
\eub{K}\sb 2(\Phi\sb\Lambda+\zeta)
+ W \rangle \Big],
\end{equation}
\begin{equation}\label{def-map-z}
Z(\mu,\zeta)=
\eub{A}\sb\Lambda^{-1}(1-\eub{P}\sb\Lambda)
\Big(
\mu \eub{K}\sb 1
+\epsilon^2(\Lambda+\mu)\eub{K}\sb 2
+W
\Big)(\Phi\sb\Lambda+\zeta).
\end{equation}

Pick $\Gamma\ge 1$ such that
\begin{equation}\label{def-gamma}
\Gamma
\ge 2\norm{\eub{A}\sb\Lambda^{-1}(1-\eub{P}\sb\Lambda)\eub{K}\sb 1\Phi\sb\Lambda}\sb{L\sp 2}.
\end{equation}

\begin{lemma}\label{lemma-mz-contraction}
Consider $\R\times L^2_r$
endowed with the metric
\[
\norm{(\mu,\zeta)}\sb{\Gamma}
=\Gamma\abs{\mu}+\norm{\zeta}\sb{L^2}.
\]
There is $\omega\sb 1\in(\omega\sb 0,m)$
such that
for $\omega\in(\omega\sb 1,m)$
the map
\begin{equation}\label{def-m-z}
M\times Z:\;\;\R\times L^2_r \to \R\times L^2_r,
\qquad
(\mu,\zeta)\mapsto\big(M(\mu,\zeta),Z(\mu,\zeta)\big),
\end{equation}
restricted onto the set
\[
\mathcal{B}\sb{\epsilon}
=\{(\mu,\zeta)\in\R\times L^2(\R,\C^{4})
\sothat\norm{(\mu,\zeta)}\sb{\Gamma}\le\epsilon \}
\subset\R\times L^2_r
\]
is an endomorphism and a contraction
with respect to $\norm{\cdot}\sb{\Gamma}$.
\end{lemma}

\begin{proof}
Assuming $\Gamma |\mu| + \| \zeta \|_{L^2}
\leq \epsilon < 1$, and using
Lemma~\ref{lemma-bounds-w}, we have the estimates
\[
  |M(\mu,\zeta)| \leq C \left( |\mu| \|\zeta\|_{L^2}
  +\epsilon^2(1+|\mu|)( 1 + \| \zeta \|_{L^2})
  + \| W \|_{L^\infty}  \right) \leq C \epsilon^2,
\]
\[
\begin{split}
&\| Z(\mu,\zeta) \|_{L^2}
\\
&
\leq
\frac{1}{2} \Gamma |\mu| +
  C \left( |\mu| \|\zeta\|_{L^2}
  + \epsilon^2(1 + |\mu|)(1 + \|\zeta\|_{L^2})
  + \| W \|_{L^\infty} (1 + \|\zeta\|_{L^2})
  \right)
\\
&
\qquad
\leq \frac{1}{2} \epsilon
  + C \epsilon^2,
\end{split}
\]
which show that for all sufficiently small $\epsilon$, $M \times Z$ maps
$\mathcal{B}_\epsilon$ into itself.
Similar estimates  show that
$(M\times Z)\at{\mathcal{B}\sb{\epsilon}}$
is a contraction
in the metric $\norm{\cdot}\sb{\Gamma}$.
\end{proof}

According to Lemma~\ref{lemma-mz-contraction},
by the contraction mapping theorem,
the map
\eqref{def-m-z}
has a unique fixed point
$
(\mu\sb 0(\omega),\zeta\sb 0(\omega))
\in
\mathcal{B}\sb{\epsilon}
\subset\R\times L^2_r
$
(as long as $\omega\in(\omega\sb 1,m)$).
Thus, we have
$
\pm\epsilon^2(\Lambda+\mu\sb 0(\omega))
\in\sigma\sb p(\eubJ\eubL(\omega)),
$
$
\omega\in(\omega\sb 1,m),
$
with $\Gamma\abs{\mu\sb 0(\omega)}\le\epsilon$,
finishing the proof of the proposition.
\end{proof}

By Remark~\ref{remark-enough},
Proposition~\ref{prop-k3}
finishes the proof of Theorem~\ref{main-theorem-dirac-vbk}.

\bibliographystyle{mfo-doi}

\bibliography{all}
\end{document}